\documentclass[11pt]{article}

\usepackage{amsthm,amsmath,amsfonts,amssymb,graphicx,bm,color}

\begin{document}

\newtheorem{claim}{Claim}
\newtheorem{lemma}{Lemma}
\newtheorem{theorem}{Theorem}
\newtheorem{proposition}{Proposition}
\newtheorem{corollary}{Corollary}
\theoremstyle{definition}
\newtheorem{definition}{Definition}
\newtheorem{example}{Example}
\newtheorem{remark}{Remark}

\setlength{\baselineskip}{16pt}
\newcommand{\dmu}{{\mathrm{d}\mu}}
\newcommand{\dnu}{{\mathrm{d}\nu}}
\newcommand{\dx}{{\mathrm{d}x}}
\newcommand{\dbmx}{{\mathrm{d}\bm{x}}}
\newcommand{\aka}[1]{\textcolor{red}{\bf #1}}

\newcommand{\indep}{\mathop{\bot\!\!\!\bot}}

\title{Inconsistency of diagonal scaling under high-dimensional limit: a replica approach}
\author{Tomonari Sei\footnote{Graduate School of Information Science and Technology, The University of Tokyo.}}
\date{August 17, 2018}
\maketitle

\begin{abstract}
In this note, we claim that diagonal scaling
of a sample covariance matrix is asymptotically inconsistent
if the ratio of the dimension to the sample size converges to a positive constant,
where population is assumed to be Gaussian
with a spike covariance model.
Our non-rigorous proof relies on the replica method developed in statistical physics.
In contrast to similar results known in literature on principal component analysis,
the strong inconsistency is not observed.
Numerical experiments support the derived formulas.
\end{abstract}

\section{Main results}

Let $\bm{x}_{(1)},\ldots,\bm{x}_{(n)}$ be
independent and identically distributed
according to the $p$-dimensional Gaussian distribution with mean vector $\bm{0}$
and covariance matrix $\bm\Sigma\in\mathbb{R}^{p\times p}$.
Denote the (uncentered) sample covariance matrix by
$\bm{S}=(1/n)\sum_{t=1}^n\bm{x}_{(t)}\bm{x}_{(t)}^\top$.
We assume $n\geq p$,
which implies that $\bm{S}$ is positive definite with probability one,
unless otherwise stated.

Let $\mathbb{R}_+$ be the set of positive numbers.
By a diagonal scaling theorem established by \cite{MarshallOlkin1968}, 
there exists a unique vector $\hat{\bm{w}}\in\mathbb{R}_+^p$ such that
\begin{align}
 \hat{\bm{W}}\bm{S}\hat{\bm{W}}\bm{1} = \bm{1},
 \label{eq:diagonal-scaling}
\end{align}
where $\hat{\bm{W}}={\rm diag}(\hat{\bm{w}})$ and $\bm{1}=\bm{1}_p=(1,\ldots,1)^\top$.
In other words, all row sums of
the scaled matrix $\hat{\bm{W}}\bm{S}\hat{\bm{W}}$ are unity.
Refer to \cite{Sei2016} for an application of this fact to a rating method
of multivariate quantitative data.

Let $\bm{w}_0$ be the population counterpart of $\hat{\bm{w}}$,
which means
$\bm{W}_0\bm\Sigma\bm{W}_0\bm{1} = \bm{1}$, $\bm{W}_0={\rm diag}(\bm{w}_0)$.
If $p$ is fixed and $n\to\infty$,
a standard argument of asymptotic statistics shows that
$\hat{\bm{w}}$ converges almost surely to the true parameter $\bm{w}_0$
because $\bm{S}$ converges to $\bm\Sigma$.
However, if $p$ is getting large as well as $n$,
then the limiting behavior of $\hat{\bm{w}}$ is not obvious.
We are interested in the behavior of $\hat{\bm{w}}$
if $\alpha_p:=n/p$ converges to some $\alpha\in [1,\infty)$ as $p\to\infty$.

In principal component analysis,
this type of high-dimensional asymptotics is deeply investigated.
In particular, the angle between the first eigenvectors of $\bm{S}$ and $\bm\Sigma$
converges to a non-zero value.
Furthermore, the limit becomes $\pi/2$ if $\alpha$ is less than a threshold.
We call these phenomena inconsistency and strong inconsistency, respectively.
The fact is found by \cite{BiehlMietzner1993,BiehlMietzner1994} in literature of statistical physics and then mathematically proved by \cite{JohnstoneLu2004,Nadler2008,Paul2007}.

We obtain similar conclusions for the diagonal-scaling problem, at least numerically, as follows.
First consider the simplest case $\bm\Sigma=\bm{I}=\bm{I}_p$,
the identity matrix of order $p$.
It is easy to see that $\bm{w}_0=\bm{1}$ for this case.
The following claim is derived in Section~\ref{section:formal-proof}
with the help of the replica method in statistical physics.

\begin{claim} \label{claim:Omega-zero}
 Let $\bm\Sigma=\bm{I}$.
 Suppose that $\alpha_p=n/p$ converges to some $\alpha\in[1,\infty)$ as $p\to\infty$.
 Then we have
 \begin{align}
  \lim_{p\to\infty} \frac{\hat{\bm{w}}^\top\bm{w}_0}{\|\hat{\bm{w}}\|\|\bm{w}_0\|}
  = \frac{1-\frac{3}{8\alpha}}{\sqrt{1-\frac{1}{2\alpha}}}.
  \label{eq:result-Omega-zero}
 \end{align}
The right hand side falls within $(5\sqrt{2}/8,1)$.
\end{claim}

The quantity $\hat{\bm{w}}^\top\bm{w}_0/(\|\hat{\bm{w}}\|\|\bm{w}_0\|)$ is
the cosine of the angle between $\hat{\bm{w}}$ and $\bm{w}_0$,
referred to as the cosine similarity.
It follows from (\ref{eq:result-Omega-zero}) that the cosine similarity
does not converge to $1$ and hence $\hat{\bm{w}}$ is inconsistent.
In contrast to principal component analysis, $\hat{\bm{w}}$ is never strongly inconsistent.
This is not a direct consequence of positivity of $\hat{\bm{w}}$ and $\bm{w}_0$.
For example, the angle between two positive vectors
$(p,1,\cdots,1)$ and $(1,\ldots,1)$ in $\mathbb{R}^p$ converges to $\pi/2$ as $p\to\infty$.

Next consider a spike covariance model given by
\begin{align}
  \bm\Sigma = \Omega\frac{\bm{1}\bm{1}^\top}{p} + \bm{I},
  \label{eq:spike}
\end{align}
where $\Omega$ is a positive constant meaning the signal-to-noise ratio.
It is easy to see that $\bm{w}_0=(\Omega+1)^{-1/2}\bm{1}$.
The following claim is also derived in Section~\ref{section:formal-proof}.

\begin{claim} \label{claim:Omega-positive}
 Assume the spike covariance model (\ref{eq:spike}) with $\Omega>0$.
 Suppose that $\alpha_p=n/p$ converges to some $\alpha\in[1,\infty)$ as $p\to\infty$.
 Then we have
 \begin{align}
  \lim_{p\to\infty} \frac{\hat{\bm{w}}^\top\bm{w}_0}{\|\hat{\bm{w}}\|\|\bm{w}_0\|}
  &= \frac{\mu}{\sqrt{\nu}},
  \label{eq:result-Omega-positive}
 \end{align}
 where $\mu$ is the unique minimizing point of a convex function 
 \begin{align}
  g(\mu) = \frac{\mu}{2} + \frac{\Omega+1}{2\Omega(\Omega \mu+1)}
  +\frac{1}{2}\log\frac{\Omega \mu+2}{(\Omega \mu+1)(\Omega \mu+2-1/\alpha)},
  \quad \mu>0,
  \label{eq:function-g}
 \end{align}
 and
 \begin{align}
  \nu = -\Omega \mu^2 + \frac{(\Omega \mu+1)(\Omega \mu+2)}{\Omega \mu+2-\frac{1}{\alpha}}.
  \label{eq:nu}
 \end{align}
\end{claim}

Let us consider extreme cases $\Omega\to 0$, $\Omega\to\infty$ and $\alpha\to \infty$.
As $\Omega\to 0$,
the quantity $\mu/\sqrt{\nu}$ expectedly converges to the right hand side of (\ref{eq:result-Omega-zero}).
For the other two extreme cases, $\mu/\sqrt{\nu}$ converges to 1.
This consequence is natural since $\Omega\to\infty$ means that the signal is infinitely large compared to the noise,
and $\alpha\to\infty$ corresponds to the classical limit.
The proof of these statements is given in Appendix~\ref{section:extreme}.

As a final remark, we consider what happens if $n<p$.
In this case, the equation (\ref{eq:diagonal-scaling}) may not have a solution,
depending on $\bm{S}$.
If $\bm\Sigma=\bm{I}$, a result of geometric probability \cite{CoverEfron1967,Wendel1962} implies that
(\ref{eq:diagonal-scaling}) admits a solution with probability
\begin{align}
 \sum_{i=0}^{n-1} \binom{p-1}{i}\Bigl(\frac{1}{2}\Bigr)^{p-1}.
 \label{eq:geometric-probability}
\end{align}
See Appendix~\ref{section:geometric-probability} for more details.
As $p\to\infty$,
the probability converges to 0 if $\alpha<1/2$ and 1 if $\alpha>1/2$.
This may be seen as a phase transition phenomenon.

The rest of the paper is as follows.
In Section~\ref{section:formal-proof},
we derive the two claims using the replica method.
In Section~\ref{section:numerical},
we perform numerical experiments for validating the formulas
as well as studying the cases that $\bm\Sigma$ is not a spike covariance model.
Section~\ref{section:discussion} concludes with open problems.
Proofs are given in Appendices.

\section{Non-rigorous proof based on the replica method} \label{section:formal-proof}

We derive the claims stated in the preceding section
using the replica method.
We will put the replica symmetry assumption
and exchange integral and limits without justification.
The outline is similar to the case of principal component analysis
(e.g.\ Chapter 3 of \cite{Watanabe}).

\subsection{The saddle point equation}

Let $\hat{\bm{w}}\in\mathbb{R}^p$ be
the solution of (\ref{eq:diagonal-scaling}).
Then $\hat{\bm{w}}$ is the unique minimizer of a strictly convex function
\begin{align}
 H(\bm{w}) = H(\bm{w}|\bm{S}) = \sum_{i=1}^p (-\log w_i)
 + \frac{1}{2}\bm{w}^\top\bm{S}\bm{w}.
\label{eq:Hamiltonian}
\end{align}
We call $H$ the Hamiltonian.
Define the partition function by
\[
 Z(\beta|\bm{S})
 = \int_{\bm{w}\in\mathbb{R}_+^p} \exp(-\beta H(\bm{w}|\bm{S})) {\rm d}\bm{w},
 \quad \beta>0.
\]
The free energy density is defined by
\[
 f(\beta|\bm{S})
 = -\frac{1}{p\beta}\log Z(\beta|\bm{S}).
\]
In order to obtain the macroscopic variables appearing in Claim~\ref{claim:Omega-zero} and \ref{claim:Omega-positive}, we calculate
\begin{align}
 \bar{f} = \lim_{\beta\to\infty} \lim_{p\to\infty} {\rm E}_{\bm{X}}[f(\beta|\bm{S})],
 \label{eq:f-bar}
\end{align}
where ${\rm E}_{\bm{X}}$ denotes the expectation with respect to
$\bm{X}=(\bm{x}_{(1)},\ldots,\bm{x}_{(n)})$.

The replica method first calculates
${\rm E}_{\bm{X}}[Z^r(\beta|\bm{S})]$
for positive integers $r$ and then formally applies an identity
\begin{align}
 \lim_{r\to 0}\frac{\partial}{\partial r}\log {\rm E}_{\bm{X}}[Z^r(\beta|\bm{S})]
 = {\rm E}_{\bm{X}}[\log Z(\beta|\bm{S})]
 \label{eq:replica}
\end{align}
as if $r$ is a real number.
We will also put the replica symmetry assumption
and exchange integration and limits without justification.

In the following, we assume the spike covariance model (\ref{eq:spike})
including the case $\Omega=0$.
We use abbreviation $\bm{w}^\beta=\prod_i w_i^\beta$
for vectors $\bm{w}=(w_i)$.
Recall that $\alpha_p=n/p\to \alpha\in[1,\infty)$ as $p\to\infty$.

Proof of all lemmas is given in Appendix~\ref{section:lemma-proof}.

\begin{lemma} \label{lemma:1}
Let $r$ be a positive integer.
Then we have
\begin{align}
 {\rm E}_{\bm{X}}[Z^r(\beta|\bm{S})]
 = \int
 \exp\left(
 p\mathcal{T}_r(\bm{Q},\bm{m})
 \right)
 \prod_{a=1}^r \bm{w}_a^\beta {\rm d}\bm{w}_a,
 \label{eq:lemma-1}
\end{align}
where $\bm{w}_1,\ldots,\bm{w}_r\in\mathbb{R}_+^p$,
\begin{align}
\bm{Q}=\Bigl(\frac{\bm{w}_a^\top\bm{w}_b}{p}\Bigr)\in\mathbb{R}^{r\times r},
\quad \bm{m}=\Bigl(\frac{\bm{w}_a^\top\bm{1}}{p}\Bigr)\in\mathbb{R}^r,
\label{eq:macroscopic-variables}
\end{align}
and the function $\mathcal{T}_r$ is defined by
\begin{align}
 \mathcal{T}_r(\bm{Q},\bm{m})
 &= -\frac{\alpha_p}{2}\log \left|\bm{I}+\frac{\beta}{\alpha_p}(\Omega\bm{m}\bm{m}^\top+\bm{Q})\right|.
 \label{eq:T_r}
 \end{align}
\end{lemma}

The quantities in Eq.~(\ref{eq:macroscopic-variables}) are macroscopic variables of interest.
By applying the Fourier inversion and saddle point approximation to (\ref{eq:lemma-1}), we obtain the following lemma.

\begin{lemma} \label{lemma:Fourier-saddle}
Let $r$ be a positive integer.
Then we have
\begin{align}
 \lim_{p\to\infty}\frac{1}{p}\log{\rm E}_{\bm{X}}[Z(\beta|\bm{S})^r]
 &= \sup_{\bm{Q},\bm{m}}
 \{\mathcal{S}_r(\bm{Q},\bm{m})
 + \mathcal{T}_r(\bm{Q},\bm{m})
 \Bigr\}
 \label{eq:optim}
\end{align}
up to an additional constant, where
\begin{align}
\mathcal{S}_r(\bm{Q},\bm{m})
&= 
 \inf_{\hat{\bm{Q}},\hat{\bm{m}}}\sup_{\bm{w}}
 \Bigl(
  \frac{1}{2}\mathop{\rm tr}(\hat{\bm{Q}}\bm{Q})
    -\hat{\bm{m}}^\top\bm{m}
 + 
   \beta\sum_{i=1}^r\log w_i
    -\frac{\bm{w}^\top \hat{\bm{Q}}\bm{w}}{2}
    +\hat{\bm{m}}^\top\bm{w}
  \Bigr)
 \label{eq:S_r}
\end{align}
and $\mathcal{T}_r$ is given in (\ref{eq:T_r}) after $\alpha_p$ is replaced with $\alpha$.
\end{lemma}

Since the optimization problem (\ref{eq:optim}) is not easy to solve,
we put the replica symmetry assumption:
suppose that the extremal point satisfies
\[
 Q_{ab}
 =\begin{cases}
  Q& \mbox{if}\ a=b,\\
  q& \mbox{if}\ a\neq b,
 \end{cases},
 \quad
  m_a = m,
\]
where $Q>0$, and
 \[
 \hat{Q}_{ab} = \begin{cases}
   \hat{Q}& \mbox{if}\ a=b\\
   -\hat{q}& \mbox{if}\ a\neq b
  \end{cases},
  \quad \hat{m}_a=\hat{m},
 \]
where $\hat{Q}>0$.
Under these assumptions, the optimal $\bm{w}$ is also written as $w_a=w$.

\begin{lemma} \label{lemma:replica-S_r-T_r}
 Under the replica symmetry, we have
\begin{align*}
 \mathcal{T}_r(\bm{Q},\bm{m}) &= -\frac{r\alpha}{2}\log\Bigl(1+\frac{\beta}{\alpha}(Q-q)\Bigr)
  -\frac{\alpha}{2}\log\Bigl(1+\frac{\beta r(\Omega m^2+ q)}{\alpha(1+\frac{\beta}{\alpha}(Q-q))}
  \Bigr)
\end{align*}
and
\begin{align*}
 \mathcal{S}_r(\bm{Q},\bm{m})
 & = \inf_{\hat{Q},\hat{q},\hat{m}}
 \Bigl\{
 \frac{r}{2}\hat{Q}Q - \frac{r(r-1)}{2}\hat{q}q
  -r\hat{m}m
  +\frac{r}{2}\log(2\pi)
  -\frac{r}{2}(\hat{Q}+\hat{q}-r\hat{q})w^2
  \\
 &\quad\quad +r\hat{m}w
  +r\beta\log w
  -\frac{r}{2}\log(\hat{Q}+\hat{q}+\frac{\beta}{w^2})
  -\frac{1}{2}\log\Bigl(1-\frac{r\hat{q}}{\hat{Q}+\hat{q}+\frac{\beta}{w^2}}\Bigr)
  \Bigr\},
  \end{align*}
where $w$ is the unique positive root of the quadratic equation
\begin{align}
  -(\hat{Q}+\hat{q}-r\hat{q})w + \hat{m} + \frac{\beta}{w} = 0.
  \label{eq:w-quad}
\end{align}
\end{lemma}

Our goal is to calculate $\bar{f}$ in (\ref{eq:f-bar}).
Using the replica trick (\ref{eq:replica})
and exchanging limits, we have
\begin{align*}
-\bar{f}
 &= \lim_{\beta\to\infty}\lim_{r\to 0}\sup_{Q,q,m}
\Bigl(
\frac{1}{\beta}\frac{\partial\mathcal{S}_r}{\partial r}
+\frac{1}{\beta}\frac{\partial\mathcal{T}_r}{\partial r}
\Bigr)
\end{align*}
We scale the variables as
\[
Q-q=\frac{\chi}{\beta}
\]
according to \cite{Watanabe}.
Then the free variables are $Q,m$ and $\chi$.

After some calculation, we obtain the following equation.
See Appendix for details.

\begin{lemma}\label{lemma:f-bar}
 Under the assumptions mentioned above, we have
 \begin{align*}
-\bar{f} &=
\sup_{Q,m,\chi}\inf_w
\Bigl\{
 - \frac{\Omega m^2+Q}{2(1+\frac{\chi}{\alpha})}
  +\frac{1}{\chi}\Bigl(\frac{Q}{2}-wm+\frac{w^2}{2}\Bigr)
   -\frac{Q}{2w^2}+\frac{2m}{w}-\frac{3}{2}+\log w
\Bigr\}.
 \end{align*}
\end{lemma}

Furthermore, we formally exchange the supremum and infimum,
and then rescale the variables as
\begin{align}
 Q=w^2\nu,\quad m=w\mu,
 \quad \chi=w^2\eta.
 \label{eq:rescale-Q-m}
\end{align}
Then
\begin{align*}
 -\bar{f} &= \inf_{w}\sup_{Q,m,\chi}
 \Bigl\{
 -\frac{\Omega m^2+Q}{2(1+\frac{\chi}{\alpha})}
 + \frac{1}{\chi}\Bigl(\frac{Q}{2}-wm+\frac{w^2}{2}\Bigr)
  -\frac{Q}{2w^2}+\frac{2m}{w}-\frac{3}{2}+\log w
 \Bigr\}.
 \\
 &= \inf_{w}\sup_{\nu,\mu,\eta}
 \Bigl\{
 -\frac{w^2(\Omega\mu^2+\nu)}{2(1+\frac{w^2\eta}{\alpha})}
 + \frac{1}{\eta}\Bigl(\frac{\nu}{2}-\mu+\frac{1}{2}\Bigr)
  -\frac{\nu}{2}+2\mu-\frac{3}{2}+\log w
 \Bigr\}.
\end{align*}
Denote the objective function by
\[
 g(\nu,\mu,\eta,w)
 = 
 -\frac{w^2(\Omega \mu^2+\nu)}{2(1+\frac{w^2\eta}{\alpha})}
 + \frac{1}{\eta}\Bigl(\frac{\nu}{2}-\mu+\frac{1}{2}\Bigr)
  -\frac{\nu}{2}+2\mu-\frac{3}{2}+\frac{1}{2}\log (w^2).
\]

Finally, the stationary conditions are
\begin{align}
 \frac{\partial g}{\partial \nu}
 &= \frac{-w^2}{2(1+\frac{w^2\eta}{\alpha})} + \frac{1}{2\eta} - \frac{1}{2} = 0,
 \label{eq:g-nu}
 \\
 \frac{\partial g}{\partial\mu}
 &= \frac{-w^2\Omega\mu}{1+\frac{w^2\eta}{\alpha}} - \frac{1}{\eta} + 2 = 0,
 \label{eq:g-mu}
 \\
 \frac{\partial g}{\partial\eta}
 &= \frac{w^4(\Omega \mu^2+\nu)}{2\alpha(1+\frac{w^2\eta}{\alpha})^2}
 -\frac{1}{\eta^2}\Bigl(\frac{\nu}{2}-\mu+\frac{1}{2}\Bigr)
 \label{eq:g-eta}
 = 0,
 \\
 \frac{\partial g}{\partial (w^2)}
 &= -\frac{\Omega \mu^2+\nu}{2(1+\frac{w^2\eta}{\alpha})^2} + \frac{1}{2w^2} = 0.
 \label{eq:g-w2}
\end{align}

\subsection{Derivation of Claim~\ref{claim:Omega-zero} and Claim~\ref{claim:Omega-positive}}

First, assume $\Omega=0$.
Then it is immediate from (\ref{eq:g-mu}),
(\ref{eq:g-nu}), (\ref{eq:g-w2}) and (\ref{eq:g-eta}) in this order
that
\[
 \eta = \frac{1}{2},
 \quad w^2 = \frac{1}{1-\frac{1}{2\alpha}},
 \quad \nu = \frac{1}{1-\frac{1}{2\alpha}}
 \quad \mbox{and}
 \quad
 \mu = \frac{1-\frac{3}{8\alpha}}{1-\frac{1}{2\alpha}}.
\]
Hence we obtain
\begin{align*}
 \frac{\mu}{\sqrt{\nu}}
 = \frac{1-\frac{3}{8\alpha}}{\sqrt{1-\frac{1}{2\alpha}}}.
\end{align*}
This implies Claim~\ref{claim:Omega-zero}.

Next, consider the case $\Omega>0$.
By Eq.~(\ref{eq:g-nu}),
\begin{align*}
 w^2 &= \frac{\frac{1-\eta}{\eta}}{1-\frac{1-\eta}{\alpha}}.
\end{align*}
Substituting it to (\ref{eq:g-mu}) yields
\begin{align}
 \eta &= \frac{\Omega \mu+1}{\Omega \mu+2}.
 \label{eq:eta-mu}
\end{align}
Therefore $w^2$ is written in terms of $\mu$ as
\begin{align}
 w^2 &= \frac{\Omega \mu+2}{(\Omega \mu+1)(\Omega \mu + 2 -\frac{1}{\alpha})}.
 \label{eq:w2-mu}
\end{align}
From Eq.~(\ref{eq:g-w2}), we have
\begin{align*}
 -\frac{(\Omega \mu^2+\nu)(1-\frac{1-\eta}{\alpha})^2}{2} + \frac{1-\frac{1-\eta}{\alpha}}{2(\frac{1-\eta}{\eta})} = 0
\end{align*}
and
\begin{align}
 \nu &= -\Omega \mu^2 + \frac{1}{\frac{1-\eta}{\eta}(1-\frac{1-\eta}{\alpha})}
 \nonumber
 \\
 &= -\Omega \mu^2 + \frac{(\Omega \mu+1)(\Omega \mu+2)}{\Omega \mu+2-\frac{1}{\alpha}}.
 \label{eq:nu-mu}
\end{align}
Now we obtain the expression of $\nu,\eta,w^2$ in terms of $\mu$.
Finally, substitute them into $g$ to obtain
\begin{align*}
 g(\mu)
 &= \frac{\Omega \mu^2+\Omega \mu+2}{2(\Omega \mu+1)}
 -\frac{3}{2} + \frac{1}{2}\log\frac{\Omega \mu+2}{(\Omega \mu+1)(\Omega \mu+2-\frac{1}{\alpha})}
 \\
 &= \frac{\mu}{2} + \frac{\Omega+1}{2\Omega(\Omega\mu+1)}
 + \frac{1}{2}\log\frac{\Omega \mu+2}{(\Omega \mu+1)(\Omega \mu+2-\frac{1}{\alpha})}
 + {\rm (const.)}.
\end{align*}
This function is well defined and strictly convex whenever $\Omega>0$ and $\alpha\geq 1$
since a function $x\mapsto\log x-\log(x-b)$ is convex if $b>0$.
The minimizer of $g$ exists 
since $g(\mu)\simeq \mu/2\to\infty$ as $\mu\to\infty$
and $g'(0)<0$.

Now Claim~\ref{claim:Omega-positive} follows.

\section{Numerical experiments} \label{section:numerical}

We numerically compute the cosine similarity between $\hat{\bm{w}}$ and $\bm{w}_0$ under various conditions.
Denote the ratio $n/p$ by $\alpha=\alpha_p$ for simplicity.

Figure~\ref{fig:alpha-profile} shows the $\alpha$-profile of the cosine similarity
under the Gaussian spike model (\ref{eq:spike})
for $\Omega\in\{0,1,10,100\}$ and $p=100$.
The number of simulation is 100 for each set of parameters.
Although we see that the simulated values are close to the theoretical curve,
there are some gaps for small $\alpha$ if $\Omega=1$ and $\Omega=10$.
The gap is not so large
if we focus on the other macroscopic variables $Q=\hat{\bm{w}}^\top\hat{\bm{w}}/p$
and $m=\hat{\bm{w}}^\top\bm{1}/p$.
See Figure~\ref{fig:Q-m}.
Note that the cosine similarity is equal to $m/\sqrt{Q}$.

We assumed $\alpha\geq 1$
at the beginning of the paper to make the sample covariance matrix $\bm{S}$ positive definite.
However, the equation (\ref{eq:diagonal-scaling}) can admit a solution
even if $\bm{S}$ is not positive definite.
In fact, the solution exists
if and only if $\bm{S}$ is strictly copositive \cite{MarshallOlkin1968},
meaning that $\bm{w}^\top\bm{S}\bm{w}>0$
for any non-negative vector $\bm{w}\neq \bm{0}$.
Figure~\ref{figure:ok} shows the probability that $\bm{S}$ is strictly copositive
for various $\alpha>0$ and $\Omega>-1$.
Note that the spike covariance model (\ref{eq:spike}) is positive definite
even for $-1<\Omega<0$.
The probability tends to 1 as $p\to\infty$ if $\alpha$ is greater than a threshold.
The threshold is lower if $\Omega$ is larger.
If $\Omega=0$, the result is consistent with the formula (\ref{eq:geometric-probability}).

In Figure~\ref{fig:outside}, we plot the cosine similarity
as a function of $\Omega>-1$ for $\alpha=1$ and $\alpha=0.7$.
It is observed that the similarity increases as $\Omega$ tends to $-1$.
This phenomenon is expected since the diagonal scaling problem (\ref{eq:diagonal-scaling})
is essentially the same as that of the inverse matrix.
More precisely, if $\hat{\bm{W}}$ solves $\hat{\bm{W}}\bm{S}\hat{\bm{W}}\bm{1}=\bm{1}$,
then it also satisfies $\hat{\bm{W}}^{-1}\bm{S}^{-1}\hat{\bm{W}}^{-1}\bm{1}=\bm{1}$.

We examine other covariance models
\begin{align}
 \bm\Sigma = \frac{1}{p}\bm{1}\bm{1}^\top
 + \bm{Q}\bm{U}\mathrm{diag}(\bm{h})\bm{U}^\top\bm{Q}^\top,
 \label{eq:other-models}
\end{align}
where $\bm{Q}\in\mathbb{R}^{p\times (p-1)}$ is a fixed matrix with properties
$\bm{Q}^\top\bm{1}=\bm{0}$ and $\bm{Q}^\top\bm{Q}=\bm{I}$,
$\bm{U}$ is a random rotation matrix of order $p-1$
and $\bm{h}$ is a positive vector in $\mathbb{R}^{p-1}$ given below.
Note that $\bm\Sigma$ becomes the identity matrix if $\bm{h}=(1,\ldots,1)$.
We define the power-law model by
\begin{align}
 \bm{h} \propto \Bigl(1,\frac{1}{2},\ldots,\frac{1}{p-1}\Bigr)
 \label{eq:power-model}
\end{align}
and the stepwise model by
\begin{align}
 \bm{h} \propto \Bigl(\underbrace{1,\ldots,1}_{p/2},\underbrace{\frac{1}{2},\ldots,\frac{1}{2}}_{p/2-1}\Bigr),
 \label{eq:stepwise-model}
\end{align}
where the proportional constant is determined to impose $\mathrm{tr}(\bm\Sigma)=p$.
Figure~\ref{fig:other} shows the cosine similarity under these models.
We generated a random rotation matrix $\bm{U}$ in (\ref{eq:other-models})
each time when $\bm{S}$ is sampled.
The replica solution (\ref{eq:result-Omega-zero}) for the identity covariance matrix
is surprisingly well fitted for the two cases.

Figure~\ref{fig:t-distribution} shows the cosine similarity
when the Gaussian distribution is replaced with the standardized $t$-distribution with 3 degrees of freedom.
The similarity is slightly smaller than the Gaussian case but the difference is not drastic.

\section{Discussion} \label{section:discussion}

In this paper, we analytically and numerically investigated
the diagonal scaling problem (\ref{eq:diagonal-scaling}) under the limit $n/p\to \alpha$.
In particular, it is claimed that the angle between the estimated vector $\hat{\bm{w}}$ and the true vector $\bm{w}_0$ does not converge to zero.
The replica solution fits the numerical experiments
except for small $\alpha$, as observed in Figure~\ref{fig:alpha-profile}.
The difference may be caused by the replica symmetry breaking (e.g.\ \cite{Nishimori}).
We have to fill the gap and rigorously prove the claims.
It is worth mentioning that the behavior of $\hat{\bm{w}}$ was relatively stable
with respect to change of probabilistic assumptions,
as seen in Figure~\ref{fig:other} and Figure~\ref{fig:t-distribution}.
On the other hand, it may be possible to consistently estimate $\bm{w}_0$
under some sparsity assumptions,
as discussed for principal component analysis in \cite{JohnstoneLu2009}.

We could not establish analytical expressions for the cases $-1<\Omega<0$.
The formula in Claim~\ref{claim:Omega-positive} is not extrapolated to the region.
Formulas for $\alpha<1$ are needed as well.
How to deal with the case of $\alpha=0$, called the high-dimensional low sample size data,
is highly non-trivial for the diagonal scaling problem.
Refer to \cite{JungMarron2009,YataAoshima2009} for this direction on principal component analysis.

Finally, although we focused only on convergence of the cosine similarity,
the limit distribution of $\{w_i\}_{i=1}^p$ like the Marchenko--Pastur law are also of interest.

\section*{Acknowledgments}

This study was supported by Kakenhi Grant Numbers
JP17K00044 and JP26108003.

\bibliographystyle{plain}
\bibliographystyle{elsarticle-num}
\bibliographystyle{elsarticle-harv}
\bibliography{ogi}

\begin{thebibliography}{19}


\bibitem{BiehlMietzner1993}
Biehl, M.\ and Mietzner, A.\ (1993).
Statistical mechanics of unsupervised learning,
{\em Europhys.\ Let.}, {\bf 24} (5), 421--426.

\bibitem{BiehlMietzner1994}
Biehl, M.\ and Mietzner, A.\ (1994).
Statistical mechanics of unsupervised structure recognition,
{\em J.\ Phys.\ A: Math.\ Gen.}, {\bf 27}, 1885--1897.

\bibitem{CoverEfron1967}
Cover, T.~M.\ and Efron, B.\ (1967).
Geometrical probability and random points on a hypersphere,
{\em Ann.\ Math.\ Statist.}, {\bf 38}, 213--220.


\bibitem{JohnstoneLu2004}
Johnstone, I.~M.\ and Lu, A.~Y.\ (2004).
Sparse principal components analysis,
Technical Report, Stanford University, Dept.\ of Statistics, arxiv:0901.4392.

\bibitem{JohnstoneLu2009}
Johnstone, I.~M.\ and Lu, A.~Y.\ (2009).
On consistency and sparsity for principal components analysis in high dimensions,
{\em J.\ Amer.\ Statist.\ Assoc.}, {\bf 104}, 682--703.

\bibitem{JungMarron2009}
Jung, S.\ and Marron, J.~S.\ (2009).
{PCA} consistency in high-dimension, low sample size context,
{\em Ann.\ Statist.}, {\bf 37}, 4104--4130.


\bibitem{MarshallOlkin1968}
Marshall, A.~W.\ and Olkin, I.\ (1968).
 Scaling of matrices to achieve specified row and column sums.
 {\em Numer. Math.}, {\bf 12}, 83--90.
 
\bibitem{Nadler2008}
Nadler, B.\ (2008).
Finite sample approximation results for principal component analysis: a matrix perturbation approach,
{\em Ann.\ Statist.}, {\bf 36} (6), 2791--2817.

\bibitem{Nishimori}
Nishimori, H. (2001).
{\em Statistical Physics of Spin Glasses and Information Processing: An Introduction},
Oxford University Press.

\bibitem{Paul2007}
Paul, D.\ (2007).
Asymptotics of sample eigenstructure for a large dimensional spiked covariance model,
{\em Statistica Sinica}, {\bf 17}, 1617--1642.
 
\bibitem{Sei2016}
Sei, T. (2016). An objective general index for multivariate ordered data,
{\em J. Multivariate Anal.}, {\bf 147}, 247--264.



\bibitem{Watanabe}
Watanabe, S., Nagao, T., Kabashima, Y., Tanaka, T.\ and Nakajima, S.\ (2014).
{\em Mathematical Science of Random Matrices} (in Japanese), Morikita Publishing.

\bibitem{Wendel1962}
Wendel, J.~G.\ (1962).
A problem in geometric probability,
{\em Math.\ Scandinavica}, {\bf 11}, 109--111.

\bibitem{YataAoshima2009}
Yata, K.\ and Aoshima, M.\ (2009).
{PCA} consistency for non-{G}aussian data in high-dimension, low sample size context,
{\em Comm.\ Statist.\ Theory Methods}, {\bf 38}, 2634--2652.

\end{thebibliography}

\appendix
\section*{Appendix}

\section{Proof of lemmas} \label{section:lemma-proof}

\begin{proof}[Proof of Lemma~\ref{lemma:1}]
By definition of the partition function, we have
 \begin{align*}
 {\rm E}_{\bm{X}}[Z^r(\beta|\bm{S})]
 &= \int 
 {\rm E}_{\bm{X}}\left[\exp\left(-\frac{\beta}{2}\sum_{a=1}^r \bm{w}_a^\top\bm{S}\bm{w}_a\right)
 \right]
 \prod_a\bm{w}_a^\beta{\rm d}\bm{w}_a
 \\
 &= \int 
 \left({\rm E}_{\bm{x}_{(1)}}\left[\exp\left(-\frac{\beta}{2n}\sum_{a=1}^r (\bm{w}_a^\top\bm{x}_{(1)})^2\right)
 \right]\right)^n
 \prod_a\bm{w}_a^\beta{\rm d}\bm{w}_a.
 \end{align*}
 The simultaneous distribution of $\{\bm{w}_a^\top\bm{x}_{(1)}\}_{a=1}^r$
 is the $r$-dimensional Gaussian distribution with mean zero and covariance matrix
\begin{align*}
  \bm{w}_a^\top \bm\Sigma \bm{w}_b
  &= \bm{w}_a^\top (\Omega\frac{\bm{1}\bm{1}^\top}{p}+\bm{I}) \bm{w}_b
  \\
  &= p(\Omega m_am_b + Q_{ab}).
\end{align*}
Recall that $m_a=(\bm{w}_a^\top\bm{1})/p$
and $Q_{ab}=(\bm{w}_a^\top\bm{w}_b)/p$.
The expectation we need is
\begin{align*}
 {\rm E}_{\bm{x}_{(1)}}
 \left[\exp\left(-\frac{\beta}{2n}\sum_{a=1}^r (\bm{w}_a^\top\bm{x}_{(1)})^2\right)
 \right]
 & = \int_{\mathbb{R}^r} \frac{e^{
 -\frac{1}{2p}\bm{\xi}^\top (\Omega\bm{m}\bm{m}^{\top}+\bm{Q})^{-1}\bm{\xi}
 -\frac{\beta}{2n}\bm{\xi}^\top\bm{\xi}
 }
 }{(2\pi p)^{r/2}|\Omega\bm{m}\bm{m}^\top+\bm{Q}|^{1/2}}
 {\rm d}\bm\xi
 \\
 &= \left|
  \bm{I} + \frac{\beta p}{n}(\Omega\bm{m}\bm{m}^\top + \bm{Q})
 \right|^{-1/2}.
\end{align*}
Hence
\begin{align*}
 {\rm E}_{\bm{X}}[Z^r(\beta|\bm{S})]
 &= \int
 \left|
  \bm{I} + \frac{\beta p}{n}(\Omega\bm{m}\bm{m}^\top + \bm{Q})
 \right|^{-n/2}
 \prod_a \bm{w}_a^\beta{\rm d}\bm{w}_a
  \\
  &= \int \exp\left(p\mathcal{T}_r(\bm{Q},\bm{m})\right)
 \prod_a \bm{w}_a^\beta{\rm d}\bm{w}_a,
\end{align*}
which completes the proof.
\end{proof}

\begin{proof}[Proof of Lemma~\ref{lemma:Fourier-saddle}]
%
 We formally use Dirac's delta function and its Fourier representation,
 but it will be justified by Schwartz' distribution theory.
 Define $\bm{Q}_*=(\bm{w}_a^\top\bm{w}_b/p)$ and $\bm{m}_*=(\bm{w}_a^\top\bm{1}/p)$
 as functions of $\{\bm{w}_a\}$,
 whereas $\bm{Q}$ and $\bm{m}$ denote free variables.
 Then
 \begin{align*}
 I &= \int e^{p\mathcal{T}_r(\bm{Q}_*,\bm{m}_*)}\prod_a\bm{w}_a^\beta {\rm d}\bm{w}_a
 \\
 &= \int \Bigl(\int \delta(\bm{Q}-\bm{Q}_*)\delta(\bm{m}-\bm{m}_*)e^{p\mathcal{T}_r(\bm{Q},\bm{m})}{\rm d}\bm{Q}{\rm d}\bm{m}
 \Bigr)\prod_a\bm{w}_a^\beta{\rm d}\bm{w}_a
 \\
 &= C\int \Bigl(\int \Bigl( \int e^{p(\frac{1}{2}{\rm tr}(\hat{\bm{Q}}(\bm{Q}-\bm{Q}_*))-\hat{\bm{m}}^\top(\bm{m}-\bm{m}_*))}{\rm d}\hat{\bm{Q}}{\rm d}\hat{\bm{m}}
 \Bigr)e^{p\mathcal{T}_r(\bm{Q},\bm{m})}{\rm d}\bm{Q}{\rm d}\bm{m}
 \Bigr)\prod_a\bm{w}_a^\beta{\rm d}\bm{w}_a
 \\
 &= C\int \Bigl(\int e^{p(\frac{1}{2}{\rm tr}(\hat{\bm{Q}}\bm{Q})-\hat{\bm{m}}^\top\bm{m})}
 \Bigl( \int e^{p(-\frac{1}{2}{\rm tr}(\hat{\bm{Q}}\bm{Q}_*)+\hat{\bm{m}}^\top\bm{m}_*)}
 \prod_a\bm{w}_a^\beta{\rm d}\bm{w}_a
 \Bigr){\rm d}\hat{\bm{Q}}{\rm d}\hat{\bm{m}}
 \Bigr)e^{p\mathcal{T}_r(\bm{Q},\bm{m})}{\rm d}\bm{Q}{\rm d}\bm{m},
 \end{align*}
 where $C$ is a constant depending only on $p$ and $r$.
 The innermost integral is
 \begin{align*}
 &\int e^{p(-\frac{1}{2}{\rm tr}(\hat{\bm{Q}}\bm{Q}_*)+\hat{\bm{m}}^\top\bm{m}_*)}
 \prod_a\bm{w}_a^\beta{\rm d}\bm{w}_a
 \\
 &= \int e^{-\frac{1}{2}\sum_i\sum_{a,b}\hat{Q}_{ab}w_{ia}w_{ib}+\sum_i\sum_a\hat{m}_aw_{ia}}
 \prod_a\bm{w}_a^\beta{\rm d}\bm{w}_a
 \\
 &= \left[
  \int_{\mathbb{R}_+^r} e^{-\frac{1}{2}\bm{w}^\top\hat{\bm{Q}}\bm{w}+\hat{\bm{m}}^\top\bm{w}}
  \bm{w}^\beta{\rm d}\bm{w}
 \right]^p.
 \end{align*}
 Therefore
 \[
  I = C\int \Bigl(\int e^{p(\frac{1}{2}{\rm tr}(\hat{\bm{Q}}\bm{Q})-\hat{\bm{m}}^\top\bm{m})}
  \left[
  \int_{\mathbb{R}_+^r} e^{-\frac{1}{2}\bm{w}^\top\hat{Q}\bm{w}+\hat{\bm{m}}^\top\bm{w}}
  \bm{w}^\beta{\rm d}\bm{w}
 \right]^p
 {\rm d}\hat{\bm{Q}}{\rm d}\hat{\bm{m}}
 \Bigr)e^{p\mathcal{T}_r(\bm{Q},\bm{m})}{\rm d}\bm{Q}{\rm d}\bm{m}.
 \]
 Finally, by using the saddle point approximation, we obtain
 (\ref{eq:optim}) and (\ref{eq:S_r}).
\end{proof}

\begin{proof}[Proof of Lemma~\ref{lemma:replica-S_r-T_r}]
 By the assumption of replica symmetry, we have
 \begin{align*}
  \mathcal{T}_r
  &= -\frac{\alpha}{2}\log\left|\bm{I}+\frac{\beta}{\alpha}(\Omega\bm{m}\bm{m}^\top+\bm{Q})
  \right|
  \\
  &= -\frac{\alpha}{2}\log\left|\left(1+\frac{\beta}{\alpha}(Q-q)\right)(\bm{I}-\frac{\bm{1}\bm{1}^\top}{r})+\left(1+\frac{\beta(\Omega m^2r+(Q-q)+qr)}{\alpha}\right)\frac{\bm{1}\bm{1}^\top}{r}
  \right|
  \\
  &= -\frac{\alpha (r-1)}{2}\log\left(1+\frac{\beta}{\alpha}(Q-q)\right)
  -\frac{\alpha}{2}\log\left(1+\frac{\beta}{\alpha}(Q-q)+\frac{\beta r}{\alpha}(\Omega m^2+q)
  \right)
  \\
  &= -\frac{\alpha r}{2}\log\left(1+\frac{\beta}{\alpha}(Q-q)\right)
   -\frac{\alpha}{2}\log\left(1+\frac{\beta r(\Omega m^2+ q)}{\alpha(1+\frac{\beta}{\alpha}(Q-q))}\right).
 \end{align*}
 Next we evaluate $\mathcal{S}_r$ in (\ref{eq:S_r}).
The maximal point $\bm{w}$ satisfies
 \[
  -\hat{\bm{Q}}\bm{w} + \hat{\bm{m}} + \frac{\beta}{\bm{w}} = \bm{0}.
 \]
 Then we have
 \begin{align*}
  \mathcal{S}_r
  &= \inf_{\hat{\bm{Q}},\hat{\bm{m}}}
  \Bigl(\frac{1}{2}\mathop{\rm tr}(\hat{\bm{Q}}\bm{Q})
    -\hat{\bm{m}}^\top\bm{m}
    \\
    & \quad\quad
  +\frac{r}{2}\log(2\pi)-\frac{\bm{w}^\top\hat{\bm{Q}}\bm{w}}{2}
  +\hat{\bm{m}}^\top\bm{w}+\beta \sum_i\log w_i
  -\frac{1}{2}\log|\hat{\bm{Q}}+\bm{D}_{\beta/\bm{w}^2}|
  \Bigr).
 \end{align*}
 By the replica symmetry assumption, $\bm{w}$ is
 written as $\bm{w}=w\bm{1}_r$, where $w$ satisfies (\ref{eq:w-quad}).
 Note also that
 \begin{align*}
  \hat{\bm{Q}}+\bm{D}_{\beta/\bm{w}^2}
  &= (\hat{Q}+\hat{q})\bm{I} - \hat{q}\bm{1}\bm{1}^\top + \frac{\beta}{w^2}\bm{I}
  \\
  &= (\hat{Q}+\hat{q}+\frac{\beta}{w^2})(\bm{I}-\frac{\bm{1}\bm{1}^\top}{r})
  + (\hat{Q}+\hat{q}+\frac{\beta}{w^2} - r\hat{q})\frac{\bm{1}\bm{1}^\top}{r}
 \end{align*}
 and thus
 \begin{align*}
  \log|\hat{\bm{Q}}+\bm{D}_{\beta/\bm{w}^2}|
  &= (r-1)\log(\hat{Q}+\hat{q}+\frac{\beta}{w^2})
  + \log(\hat{Q}+\hat{q}+\frac{\beta}{w^2}-r\hat{q})
  \\
  &= r\log(\hat{Q}+\hat{q}+\frac{\beta}{w^2})
  + \log\Bigl(1 - \frac{r\hat{q}}{\hat{Q}+\hat{q}+\frac{\beta}{w^2}}
  \Bigr).
 \end{align*}
 Then we obtain
 \begin{align*}
 \mathcal{S}_r
 & = \inf_{\hat{Q},\hat{q},\hat{m}}\Bigl(
  \frac{r}{2}\hat{Q}Q - \frac{r(r-1)}{2}\hat{q}q
  -r\hat{m}m
  +\frac{r}{2}\log(2\pi)
  -\frac{r}{2}(\hat{Q}+\hat{q}-r\hat{q})w^2
  \\
  &\quad\quad +r\hat{m}w
  +r\beta\log w
  -\frac{r}{2}\log(\hat{Q}+\hat{q}+\frac{\beta}{w^2})
  -\frac{1}{2}\log(1-\frac{r\hat{q}}{\hat{Q}+\hat{q}+\frac{\beta}{w^2}})
 \Bigr).
 \end{align*}
 This completes the proof.
\end{proof}

\begin{proof}[Proof of Lemma~\ref{lemma:f-bar}]
We calculate
\[
 -\bar{f}
  = \sup_{Q,q,m} \left(\frac{1}{\beta}\frac{\partial\mathcal{S}_r}{\partial r}
 + \frac{1}{\beta}\frac{\partial\mathcal{T}_r}{\partial r}
  \right)
 \Bigr|_{r\to 0} \Bigr|_{\beta\to\infty}.
\]
Since $\hat{m}$ and $w$ have a one-to-one correspondence
by (\ref{eq:w-quad}),
we can use $w$ as an free variable
and determine $\hat{m}$ by
\[
 \hat{m} =  (\hat{Q}+\hat{q}-r\hat{q})w - \frac{\beta}{w}.
\]
We also rescale the variables in $\mathcal{S}_r$ as
\[
 Q-q=\frac{\chi}{\beta},
 \quad \hat{Q}+\hat{q}=E\beta,
 \quad \hat{q}=F\beta^2.
\]
 From Lemma~\ref{lemma:replica-S_r-T_r}, we obtain
 \begin{align*}
  \frac{\partial\mathcal{S}_r}{\partial r}
  \Bigr|_{r\to 0}
  &= \inf_{\hat{Q},\hat{q},\hat{m}}\Bigl(
  \frac{1}{2}\hat{Q}Q+\frac{1}{2}\hat{q}q-\hat{m}m + \frac{1}{2}\log(2\pi)
  -\frac{1}{2}(\hat{Q}+\hat{q})w^2
  \\
  &\quad\quad + \hat{m}w + \beta\log w
  -\frac{1}{2}\log(\hat{Q}+\hat{q}+\frac{\beta}{w^2})
 +\frac{1}{2}\frac{\hat{q}}{\hat{Q}+\hat{q}+\frac{\beta}{w^2}}
  \Bigr)
  \\
  &= \inf_{\hat{Q},\hat{q},w}\Bigl(
  \frac{1}{2}\hat{Q}Q+\frac{1}{2}\hat{q}q-((\hat{Q}+\hat{q})w-\frac{\beta}{w})m + \frac{1}{2}\log(2\pi)
  \\
  &\quad\quad
   +\frac{1}{2}(\hat{Q}+\hat{q})w^2
   - \beta + \beta\log w
  -\frac{1}{2}\log(\hat{Q}+\hat{q}+\frac{\beta}{w^2})
 +\frac{1}{2}\frac{\hat{q}}{\hat{Q}+\hat{q}+\frac{\beta}{w^2}}
  \Bigr)
   \\
  &= \inf_{w,E,F}\Bigl(
   \frac{\beta}{2}EQ - \frac{\beta}{2}F\chi
   -\beta Ewm + \frac{\beta m}{w} + \frac{1}{2}\log(2\pi)
   \\
   &\quad\quad
   +\frac{\beta}{2}Ew^2-\beta + \beta \log w
   -\frac{1}{2}\log(\beta E+\frac{\beta}{w^2})
   +\frac{1}{2}\frac{\beta F}{E+\frac{1}{w^2}}
  \Bigr).
 \end{align*}
 Furthermore,
 \begin{align*}
 &\frac{1}{\beta}\frac{\partial\mathcal{S}_r}{\partial r}
 \Bigr|_{r\to 0} \Bigr|_{\beta\to\infty}
 \\
 &= \inf_{w,E,F} \Bigl(
 \frac{EQ}{2}-\frac{F\chi}{2}-Ewm+\frac{m}{w}+\frac{Ew^2}{2}-1+\log w
 +\frac{F}{2(E+\frac{1}{w^2})}
 \Bigr)
 \\
 &= \inf_{w,E,F}\Bigl(
  E\Bigl(
   \frac{Q}{2} - wm + \frac{w^2}{2}
  \Bigr)
   -\frac{F\chi}{2} + \frac{F}{2(E+\frac{1}{w^2})}
   + \Bigl(
    \frac{m}{w}-1+\log w
   \Bigr)
 \Bigr)
 \\
 &= \inf_{w} \Bigl(
  \frac{1}{\chi}\Bigl(\frac{Q}{2}-wm+\frac{w^2}{2}\Bigr)
  + \Bigl(
  -\frac{Q}{2w^2}+\frac{2m}{w}-\frac{3}{2}+\log w
  \Bigr)
 \Bigr),
 \end{align*}
 where the last equality follows from
 \[
  E = \frac{1}{\chi} - \frac{1}{w^2}
  \quad \mbox{and} \quad
  F = \frac{2}{\chi^2}\Bigl(\frac{Q}{2}-wm+\frac{w^2}{2}\Bigr)
 \]
 at the stationary point.
 Similarly we have
\begin{align*}
 \frac{\partial\mathcal{T}_r}{\partial r}
  \Bigr|_{r\to 0}
 &= - \frac{\alpha}{2}\log(1+\frac{\chi}{\alpha})
 -\frac{\beta(\Omega m^2+q)}{2(1+\frac{\chi}{\alpha})}
 \end{align*}
and
\begin{align*}
  \frac{1}{\beta}\frac{\partial\mathcal{T}_r}{\partial r}
 \Bigr|_{r\to 0} \Bigr|_{\beta\to\infty}
 &= - \frac{\Omega m^2+Q}{2(1+\frac{\chi}{\alpha})}.
\end{align*}
Now the result follows from
\begin{align*}
 & \left(\frac{1}{\beta}\frac{\partial\mathcal{S}_r}{\partial r}
 + \frac{1}{\beta}\frac{\partial\mathcal{T}_r}{\partial r}
  \right)
 \Bigr|_{r\to 0} \Bigr|_{\beta\to\infty}
 \\
  &= \inf_{w} \Bigl(
  \frac{1}{\chi}\Bigl(\frac{Q}{2}-wm+\frac{w^2}{2}\Bigr)
  + \Bigl(
  -\frac{Q}{2w^2}+\frac{2m}{w}-\frac{3}{2}+\log w
  \Bigr)
 \Bigr)
 - \frac{\Omega m^2+Q}{2(1+\frac{\chi}{\alpha})}.
\end{align*}
\end{proof}

\section{Limiting behavior as $\Omega\to 0$, $\Omega\to\infty$ and $\alpha\to\infty$} \label{section:extreme}

First consider the case $\Omega\to 0$ for fixed $\alpha$ in Claim~\ref{claim:Omega-positive}.
The objective function $g$ in (\ref{eq:function-g}) is asymptotically written as
\[
 g(\mu) = \frac{1}{2\Omega} + {\rm (const.)}
 + \frac{\Omega}{2}\Bigl(\mu^2
 -\frac{2(1-\frac{3}{8\alpha})}{1-\frac{1}{2\alpha}}\mu
 \Bigr)
 + {\rm O}(\Omega^2)
\]
as $\Omega\to 0$.
Hence the minimizer $\mu$ converges to $(1-\frac{3}{8\alpha})/(1-\frac{1}{2\alpha})$.
The variable $\nu$ is
\begin{align*}
 \nu
 &= 
  -\Omega \mu^2 + \frac{(\Omega \mu+1)(\Omega \mu+2)}{\Omega \mu+2-\frac{1}{\alpha}}
  \\
 & =\frac{1}{1-\frac{1}{2\alpha}} + {\rm O}(\Omega).
\end{align*}

Next consider the case $\Omega\to \infty$ for fixed $\alpha$.
We show that both $\mu$ and $\nu$ converge to 1.
Since the objective function is asymptotically
\begin{align*}
 g(\mu)
 &= \frac{\mu}{2} + \frac{1/\alpha}{2\Omega \mu} - \frac{1}{2}\log \mu - \frac{1}{2}\log\Omega + {\rm O}\Bigl(\frac{1}{\Omega^2}\Bigr)
\end{align*}
as $\Omega\to\infty$,
the stationary condition is
\begin{align*}
0 = g'(\mu) = \frac{1}{2} - \frac{1/\alpha}{2\Omega \mu^2} - \frac{1}{2\mu} + {\rm O}\Bigl(\frac{1}{\Omega^2}\Bigr).
\end{align*}
Solving this equation, we obtain
\[
 \mu = 1 + \frac{1}{\Omega\alpha} + {\rm O}\Bigl(\frac{1}{\Omega^2}\Bigr).
\]
By (\ref{eq:nu-mu}), we have
\begin{align*}
\nu &= -\Omega \mu^2 + \Omega \mu
\left(1 + \frac{1+1/\alpha}{\Omega \mu}
\right)
+{\rm O}\Bigl(\frac{1}{\Omega}\Bigr)
\\
&= \Omega \mu(1-\mu) + 1 + \frac{1}{\alpha}
+{\rm O}\Bigl(\frac{1}{\Omega}\Bigr)
\\
&= \Omega \Bigl(1 + \frac{1}{\Omega\alpha}\Bigr)\Bigl(-\frac{1}{\Omega\alpha}\Bigr) + 1 + \frac{1}{\alpha}
+{\rm O}\Bigl(\frac{1}{\Omega}\Bigr)
\\
&= 1
+{\rm O}\Bigl(\frac{1}{\Omega}\Bigr).
\end{align*}

Finally, consider the case $\alpha\to\infty$ for fixed $\Omega$.
We show $\mu$ and $\nu$ converge to 1.
The objective function converges to
\begin{align*}
 g(\mu) = \frac{\mu}{2} + \frac{\Omega + 1}{2\Omega(\Omega \mu+1)} - \frac{1}{2}\log(\Omega \mu+1).
\end{align*}
The stationary condition is
\[
 g'(\mu) = \frac{1}{2} - \frac{\Omega + 1}{2(\Omega \mu+1)^2} - \frac{\Omega}{2(\Omega \mu+1)} = 0,
\]
which has the unique positive solution $\mu=1$.
Furthermore, $\nu=-\Omega \mu^2 + \Omega \mu + 1 = 1$.

\section{Proof of Eq.~(\ref{eq:geometric-probability})} \label{section:geometric-probability}

We calculate the probability of the event
that $\bm{S}$ is strictly copositive (see Section~\ref{section:numerical} for the definition).
Denote the data matrix by
\[
\bm{X}= (\bm{x}_1,\ldots,\bm{x}_p) = \begin{pmatrix}
 \bm{x}_{(1)}^\top\\
 \vdots\\
 \bm{x}_{(n)}^\top
 \end{pmatrix}.
\]
Then $\bm{S}=n^{-1}\bm{X}^\top \bm{X}$.
The following are equivalent to each other.
\begin{itemize}
\item[(a)] $\bm{S}$ is strictly copositive
\item[(b)] There is not a non-negative non-zero vector $\bm{v}$
such that $\sum_i v_i\bm{x}_i=\bm{0}$.
\item[(c)] $\bm{x}_1,\ldots,\bm{x}_p$ generates a proper convex cone in $\mathbb{R}^n$.
\end{itemize}
As stated in \cite{CoverEfron1967}, the probability of the event (c)
is given by Eq.~(\ref{eq:geometric-probability})
if $\bm{x}_i$'s are independent
and the distribution of each $\bm{x}_i$ is symmetric with respect to the origin.
This result is due to \cite{Wendel1962}
and related to Schl\"afli's theorem,
which states that $p$ hyperplanes in general position in $\mathbb{R}^n$ divide $\mathbb{R}^n$
into $2\sum_{i=0}^{n-1}\binom{p-1}{i}$ regions.

\newpage

\begin{figure}[htbp]
\centering
 \begin{tabular}{cc}
 \includegraphics[width=0.48\textwidth]{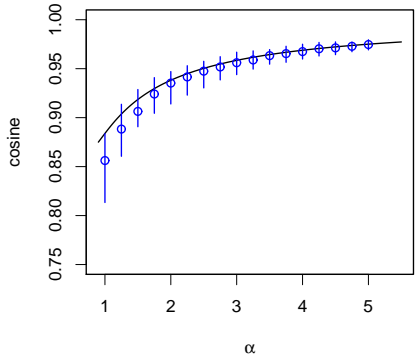}
 & \includegraphics[width=0.48\textwidth]{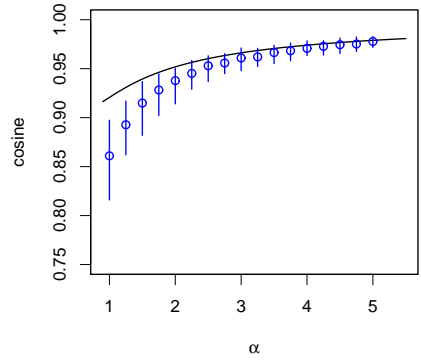}
 \\
 (a) $\Omega=0$.
 & (b) $\Omega=1$.
\\
 \includegraphics[width=0.48\textwidth]{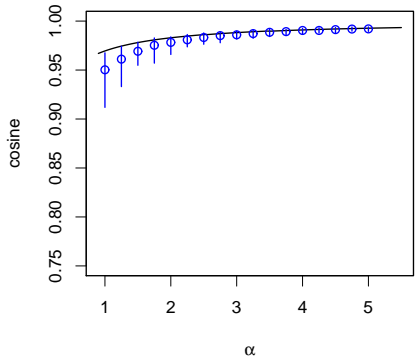}
 & \includegraphics[width=0.48\textwidth]{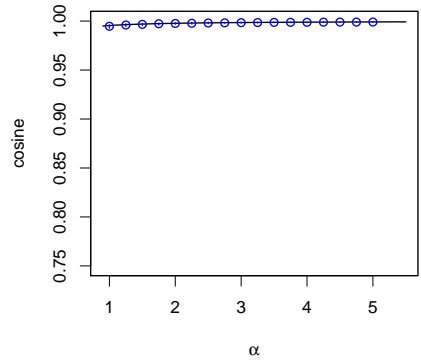}
 \\
 (c) $\Omega=10$.
 & (d) $\Omega=100$.
 \end{tabular}
 \caption{The cosine similarity of $\hat{\bm{w}}$ and $\bm{w}_0$ as a function of $\alpha$
 under the Gaussian spike covariance model with $p=100$.
 The parameter $\Omega$ is set to be (a) $\Omega=0$, (b) $\Omega=1$, (c) $\Omega=10$
 and (d) $\Omega=100$.
 The points and whisker bars indicate the median
 and 90-percent range of 100 simulated values.
 The solid line is the theoretical curve obtained in Claim~\ref{claim:Omega-zero} and Claim~\ref{claim:Omega-positive}.
 }
\label{fig:alpha-profile}
\end{figure}

\begin{figure}[htbp]
\centering
 \begin{tabular}{cc}
 \includegraphics[width=0.48\textwidth]{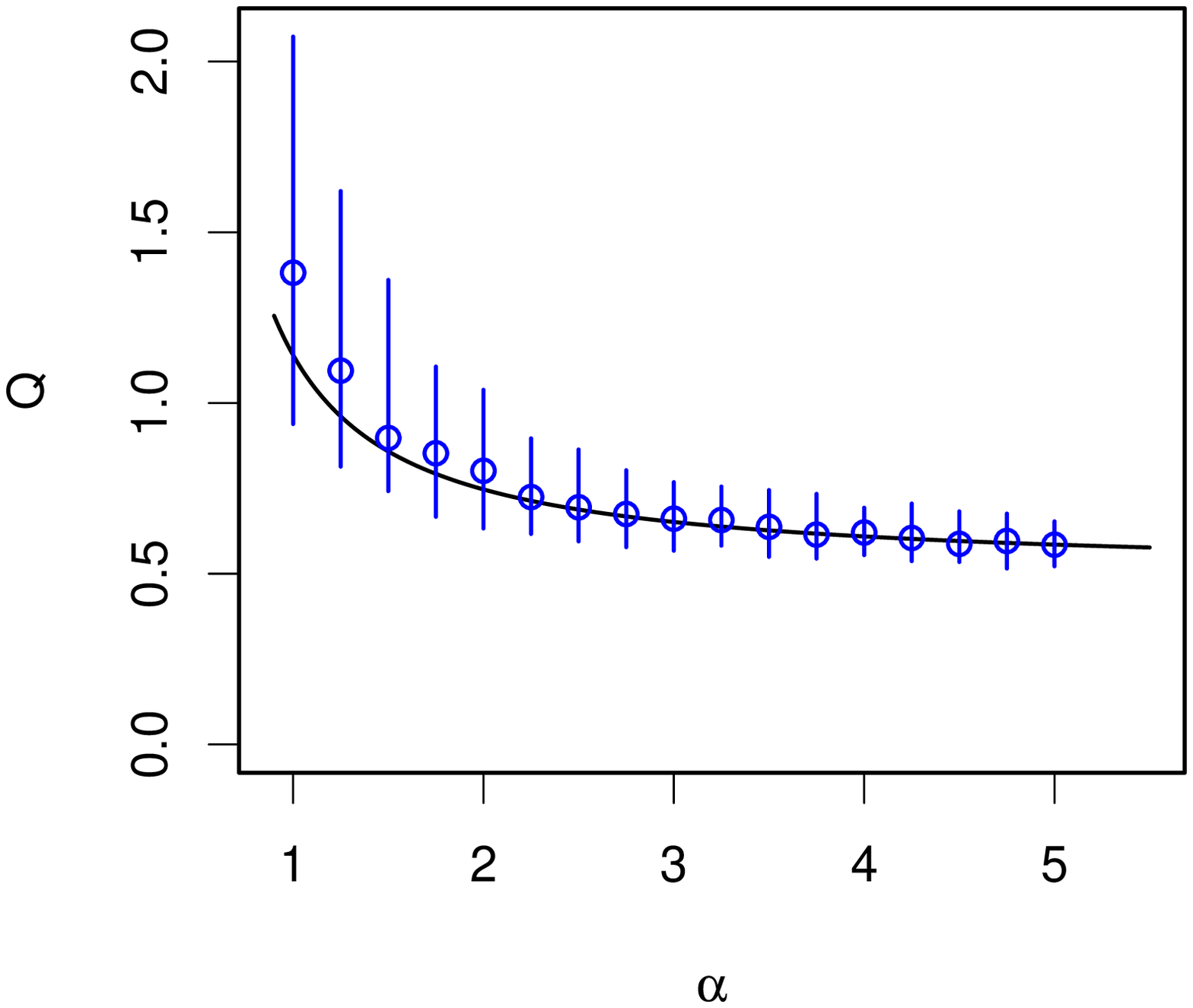}
 & \includegraphics[width=0.48\textwidth]{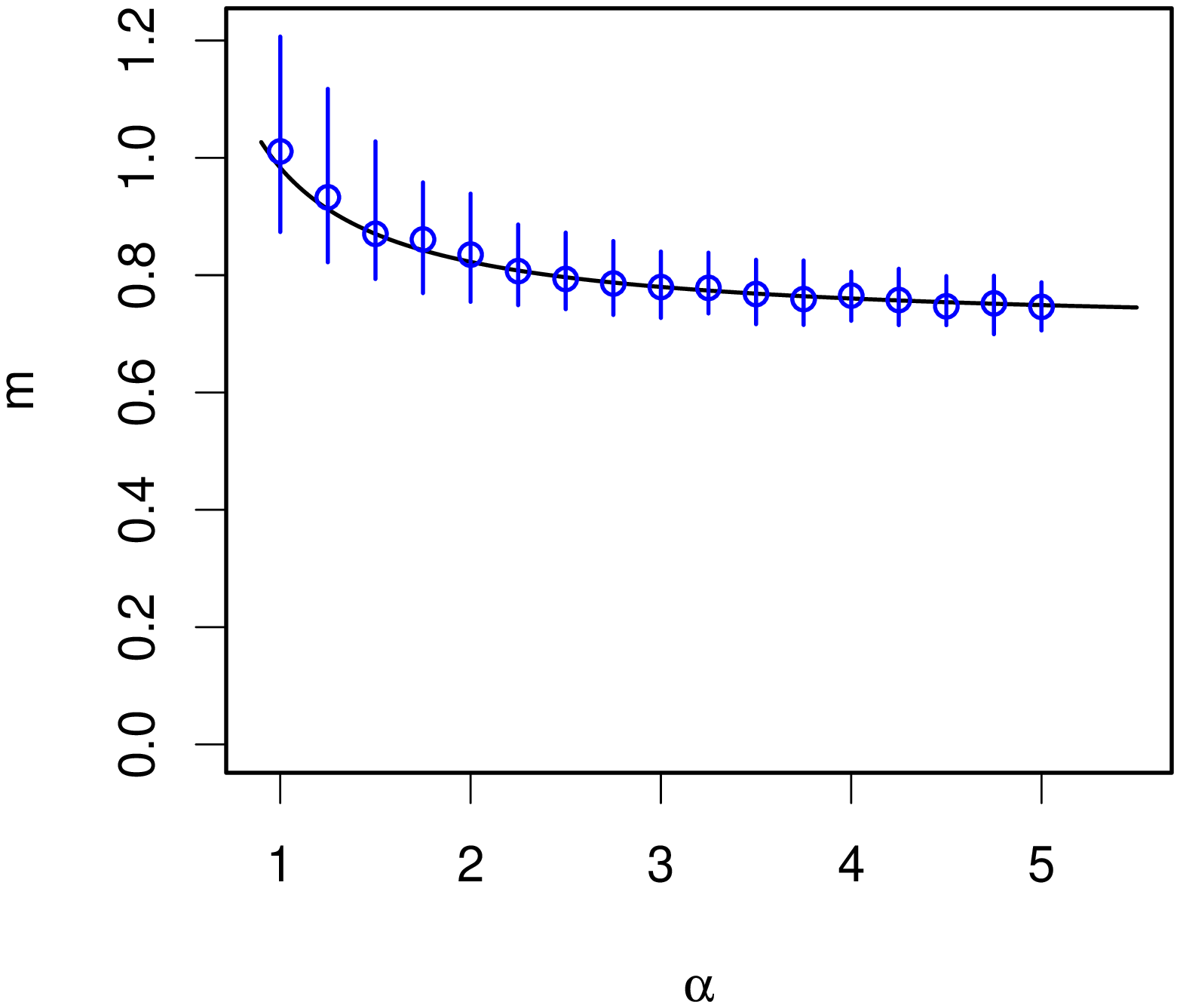}
\end{tabular}
\caption{The macroscopic variables $Q=\hat{\bm{w}}^\top\hat{\bm{w}}/p$ (left) and $m=\hat{\bm{w}}^\top\bm{1}/p$ (right) as a function of $\alpha$
 under the Gaussian spike covariance model with $p=100$ and $\Omega=1$.
 The median and 90-percent range are based on 100 simulated values.
 The solid line is the theoretical curve determined by (\ref{eq:rescale-Q-m}).
}
\label{fig:Q-m}
\end{figure}

\begin{figure}[htbp]
\centering
\begin{tabular}{cc}
 \includegraphics[width=0.48\textwidth]{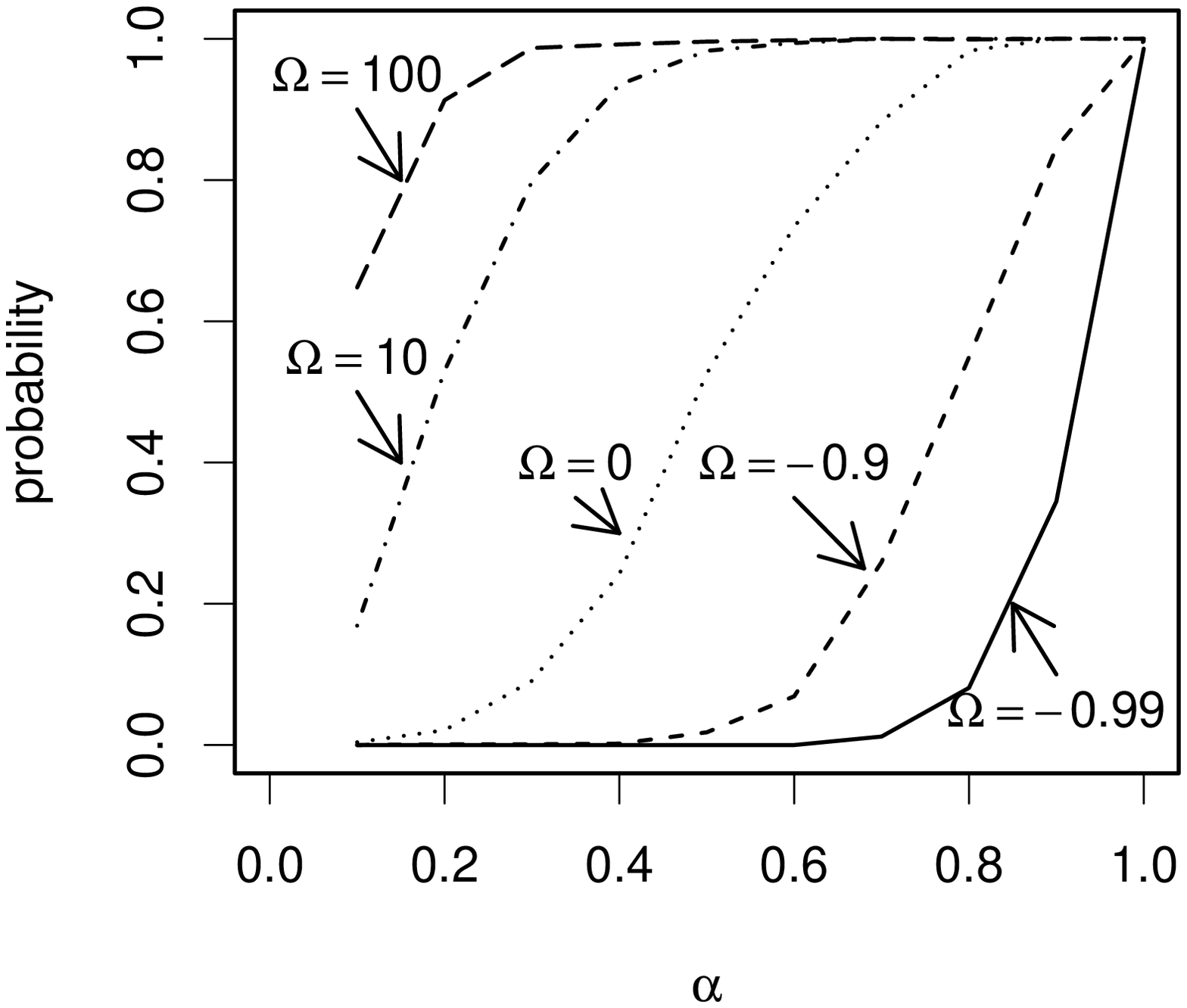}
 &  \includegraphics[width=0.48\textwidth]{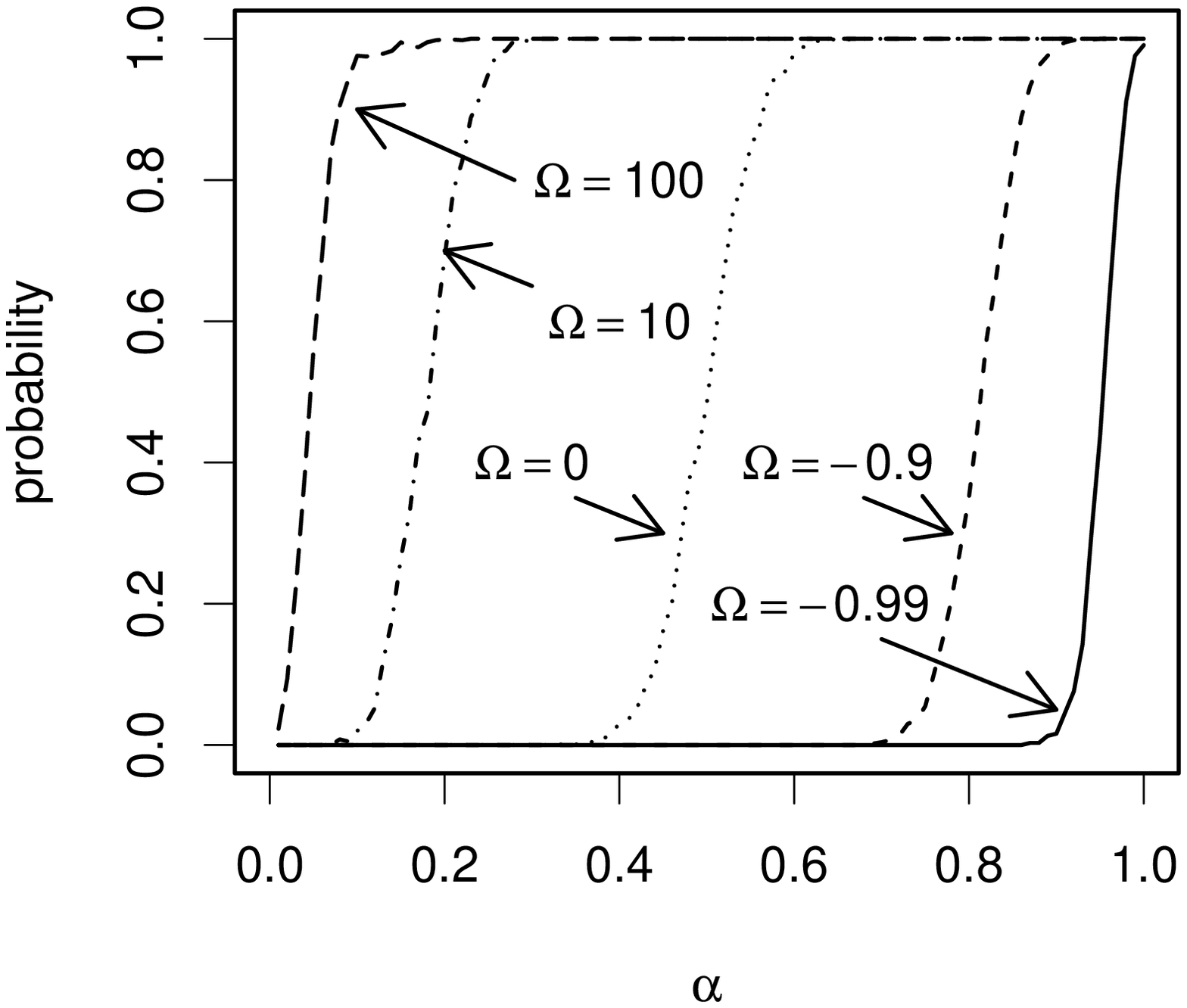}
 \\
 (a) $p=10$.
 & (b) $p=100$.
\end{tabular}
\caption{The relative frequency that Eq.~(\ref{eq:diagonal-scaling}) admits a solution.
The horizontal axis denotes $\alpha\leq 1$.
The samples are drawn from the Gaussian spike covariance model
with the parameter $\Omega\in\{-0.99,-0.9,0,10,100\}$.
The dimension is (a) $p=10$ and (b) $p=100$.
The number of simulation in each setting is $10^3$.
}
\label{figure:ok}
\end{figure}

\begin{figure}[htbp]
 \centering
 \begin{tabular}{cc}
 \includegraphics[width=0.48\textwidth]{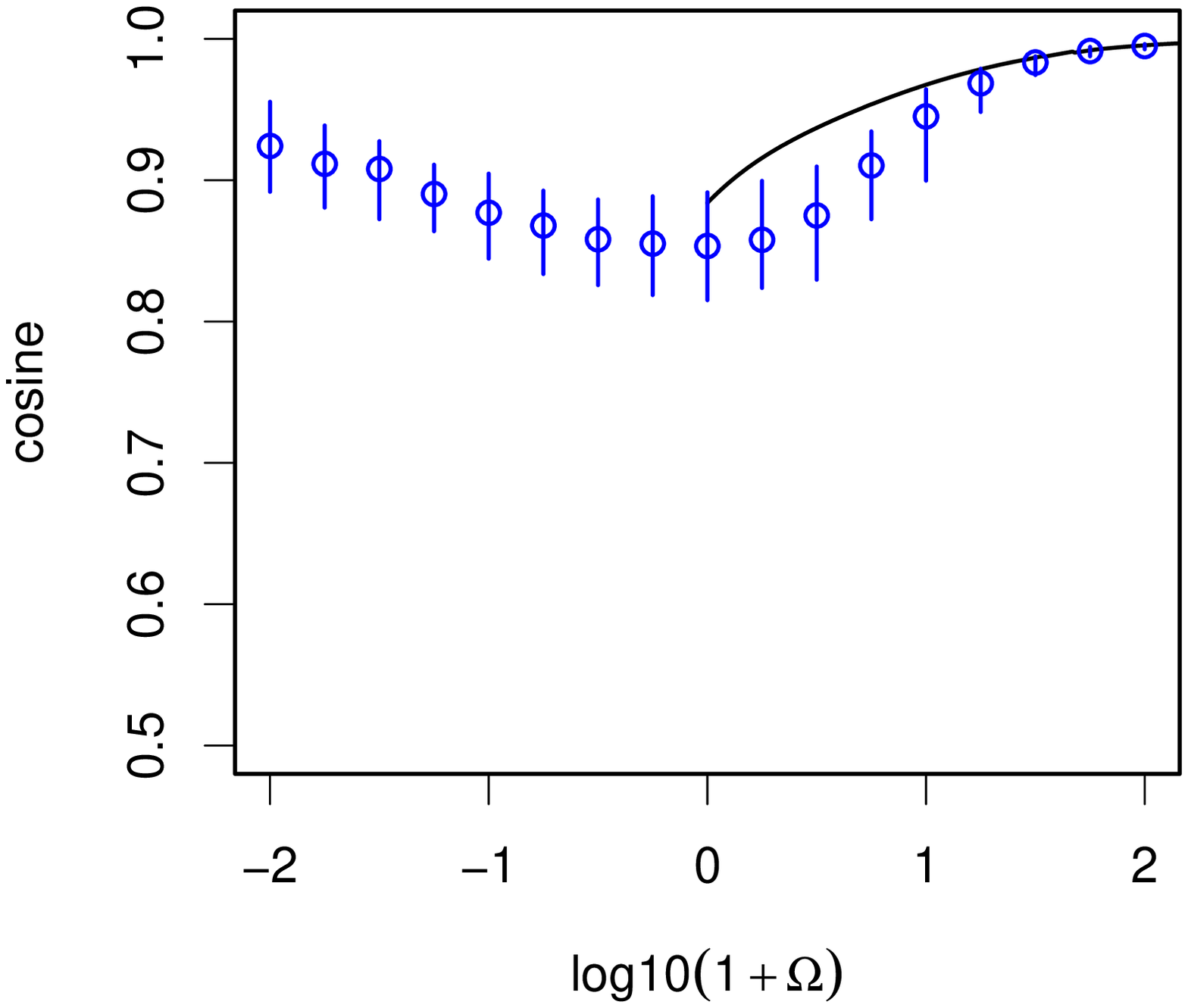}
  &  \includegraphics[width=0.48\textwidth]{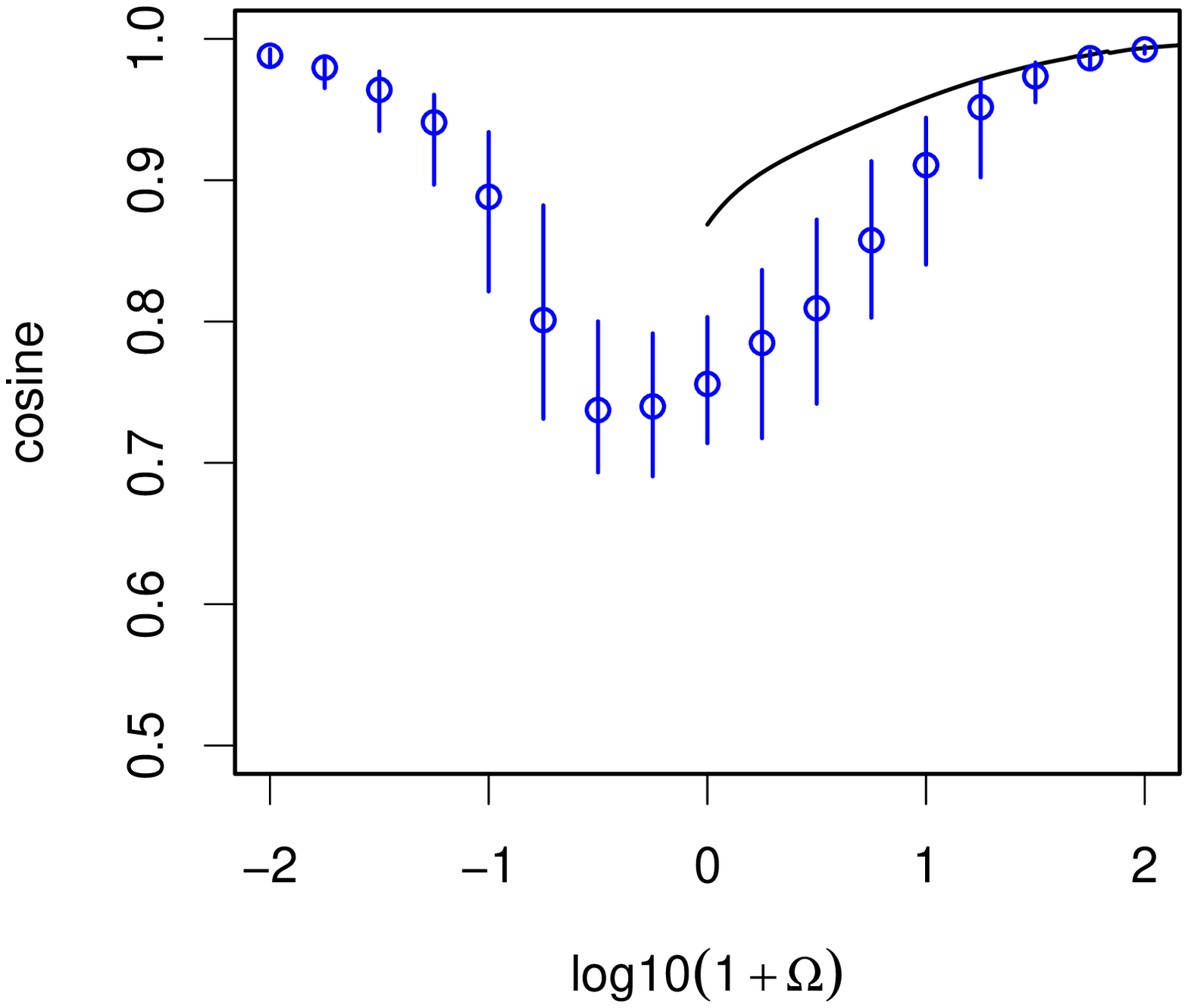}
 \\
 (a) $\alpha=1$.
 & (b) $\alpha=0.7$.
 \end{tabular}
\caption{The cosine similarity
 against $\log_{10}(1+\Omega)$ under the Gaussian spike covariance model with $p=100$,
 where the case $-1<\Omega\leq 0$ is included.
  The median and 90-percent range are based on 100 simulated values.
 The solid line is the theoretical curve for $\Omega\geq 0$.
}
\label{fig:outside}
\end{figure}

\begin{figure}[htbp]
\centering
 \begin{tabular}{cc}
 \includegraphics[width=0.48\textwidth]{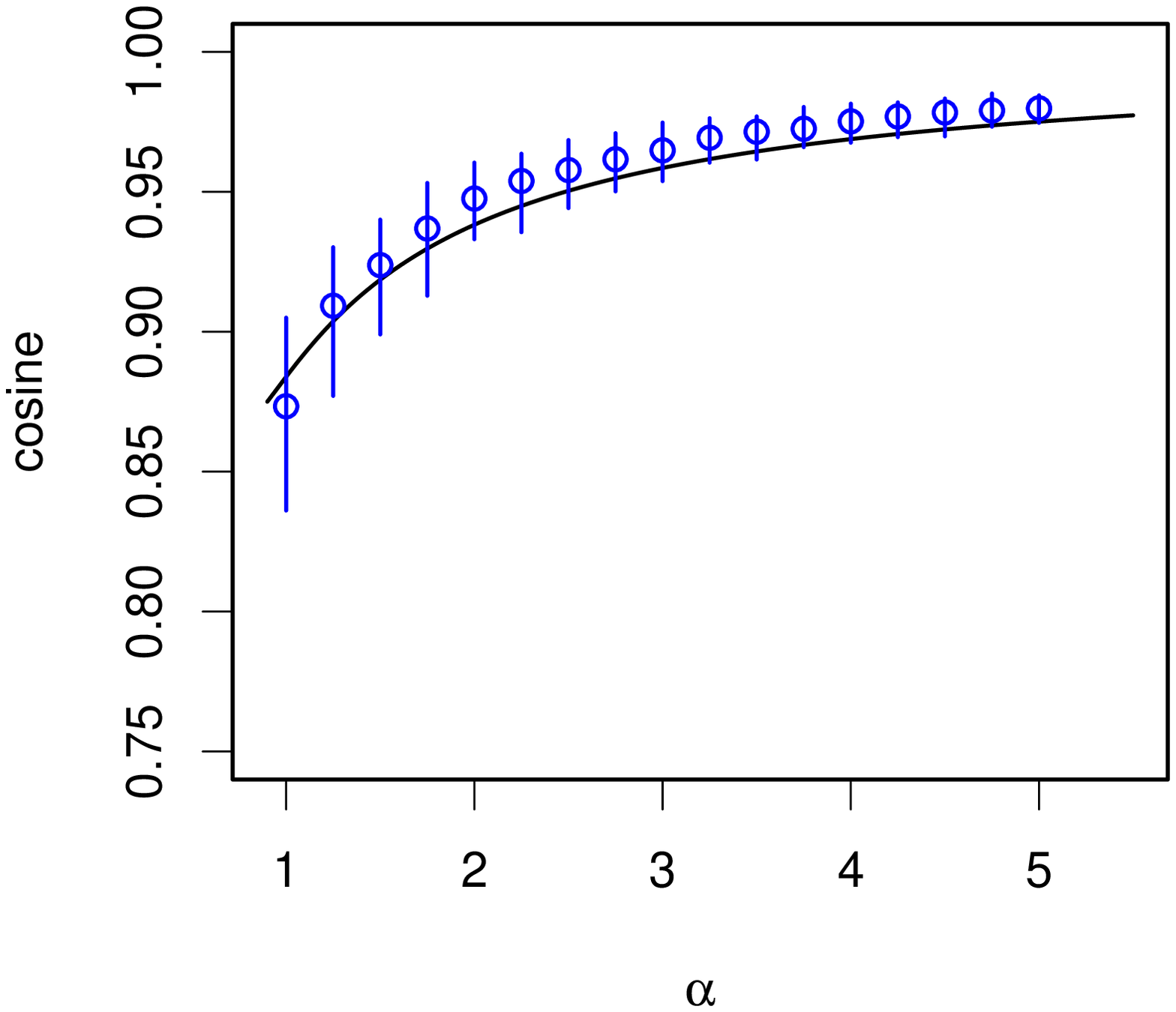}
 &  \includegraphics[width=0.48\textwidth]{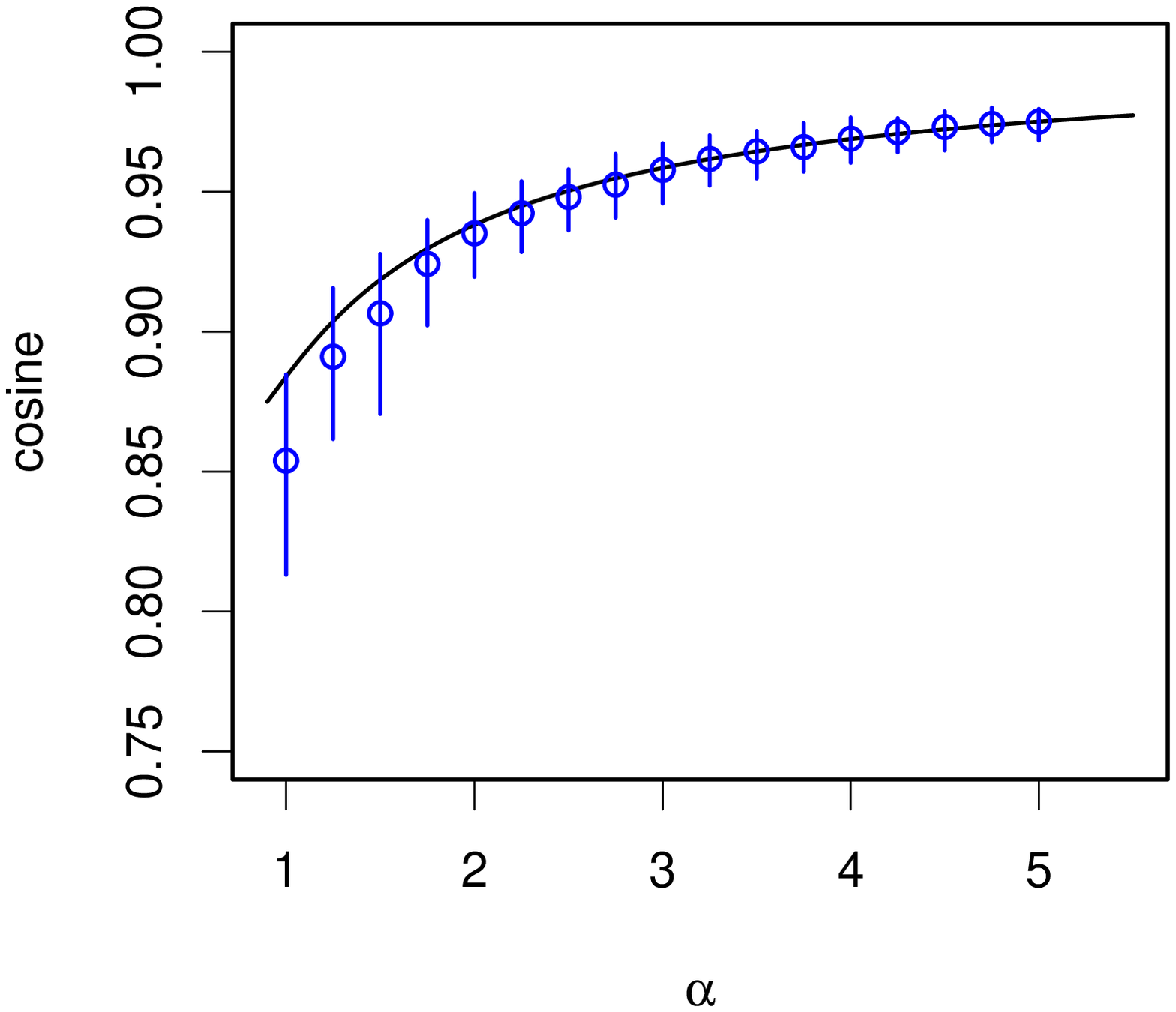}
 \\
 (a) The power-law model.
 & (b) The stepwise model.
 \end{tabular}
\caption{The cosine similarity as a function of $\alpha$
 under the power-law and stepwise models with $p=100$.
  The points and whisker bars indicate the median
 and 90-percent range of 100 simulated values.
 The solid line is the theoretical curve (\ref{eq:result-Omega-zero}) for the identity covariance.
}
\label{fig:other}
\end{figure}

\begin{figure}[htbp]
\centering
 \begin{tabular}{c}
 \includegraphics[width=0.48\textwidth]{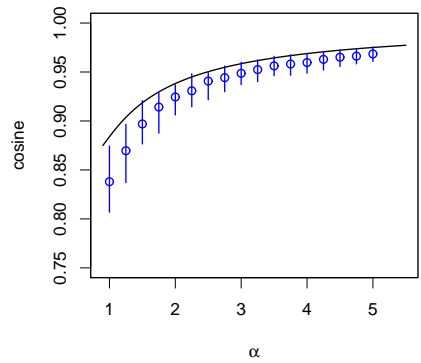}
 \end{tabular}
\caption{The cosine similarity
 as a function of $\alpha$
 when the samples are drawn from the i.i.d.\ standardized $t$-distribution
 with 3 degrees of freedom.
 The dimension is $p=100$.
  The points and whisker bars indicate the median
 and 90-percent range of 100 simulated values.
 The solid line is the theoretical curve (\ref{eq:result-Omega-zero}) for the identity covariance.
}
\label{fig:t-distribution}
\end{figure}

\end{document}